\documentclass[a4paper,10pt]{scrartcl}

\usepackage{amsmath, amstext, paralist, amsthm, amssymb, epsfig, paralist}
\usepackage{diagrams}

\newfont{\theoremfont}{cmssbx12 scaled 875}
\newtheoremstyle{Eins}{\topsep}{\topsep}{\itshape}{}{\theoremfont}{.}{5pt}{\thmname{#1}\thmnumber{ #2}\thmnote{ #3}}
\newtheoremstyle{Zwei}{\topsep}{\topsep}{}{}{\theoremfont}{.}{5pt}{\thmname{#1}\thmnumber{ #2}\thmnote{ #3}}

\theoremstyle{Zwei}
\newtheorem{thm}{Theorem}[section]
\newtheorem{lem}[thm]{Lemma}
\newtheorem{prop}[thm]{Proposition}
\newtheorem{ex}[thm]{Example}

\newtheorem{cor}[thm]{Corollary}
\newtheorem*{ack}{Acknowledgements}

\renewcommand{\bar}{\overline}
\newcommand{\bc}[2]{\genfrac{[}{]}{0pt}{}{#1}{#2}}
\newcommand{\cb}[2]{{\scriptstyle[ #1\;#2]}}
\newcommand{\Ker}{\operatorname{Ker}\nolimits}
\newcommand{\Coker}{\operatorname{Cok}\nolimits}
\newcommand{\Cok}{\operatorname{Cok}\nolimits}
\newcommand{\cok}{\operatorname{cok}\nolimits}
\newcommand{\coker}{\operatorname{cok}\nolimits}

\renewcommand{\Im}{\operatorname{Im}\nolimits}
\newcommand{\im}{\operatorname{im}\nolimits}
\newcommand{\coim}{\operatorname{coim}\nolimits}
\newcommand{\Coim}{\operatorname{Coim}\nolimits}
\newcommand{\id}{\operatorname{id}\nolimits}

\newarrow{Dash}{dash}{dash}{dash}{dash}{>}

\begin{document}
\title{On the notion of a semi-abelian category in the sense of Palamodov}
\author{Yaroslav Kopylov and Sven-Ake Wegner}
\date{March 3, 2011}
\maketitle
\renewcommand{\thefootnote}{}
\hspace{-1000pt}
\hspace{-1000pt}
\vspace{-35pt}

{\small
\begin{abstract}\noindent{}In the sense of Palamodov, a preabelian category is semi-abelian if for every morphism the natural morphism between the cokernel of its kernel and the kernel of its cokernel is simultaneously a monomorphism and an epimorphism. In this article we present several conditions which are all equivalent to semi-abelianity. First we consider left and right semi-abelian categories in the sense of Rump and establish characterizations of these notions via six equivalent properties. Then we use these properties to deduce the characterization of semi-abelianity. Finally, we investigate two examples arising in functional analysis which illustrate that the notions of right and left semi-abelian categories are distinct and in particular that such categories occur in nature.
\end{abstract}}

\section{Introduction}\label{Introduction}\vspace{-5pt}

Additive and non-abelian but abelian-like categories arise in different branches of mathematics. Quillen \cite{Quillen1973} studied categories arising in K-theory which need not have (co-)kernels. Rump \cite{Rump2001} and Bondal, van den Bergh \cite{BondalVanDenBergh} showed that torsion theories can be described by certain non-abelian categories. No reasonable subcategory of the category of locally convex spaces (including the category itself) is abelian, cf.~Prosmans \cite{Prosmans2000} or Wengenroth \cite{Wengenroth}. In the context of bounded cohomology (see Monod \cite{Monod1, Monod2}), non-abelian additive categories have been used systematically by B\"uhler \cite{Buehler2008}. In Lie theory, non-abelian categories appear in the framework of the so-called equivariant derived category (see e.g.~Kashiwara \cite{Kashiwara}).\vspace{3pt}
\\A priori all categories above are not accessible for standard methods of homological algebra since there might be different choices to define \textquotedblleft{}(short) exact sequences\textquotedblright{} and it is moreover not clear if basic results like the snake, five, horseshoe, or comparison lemmas remain true in the given category (cf.~Grandis \cite {Grand91-1}, Kopylov, Kuz$'$minov \cite{KoK00, KoK09} or Kopylov \cite{Kopylov2009}). On the other hand, the reason for considering the categories above is in fact the hope that a certain, for example purely analytic, problem can be formulated in the language of exact sequences and then eventually be solved with the help of purely abstract methods of homological algebra.
\smallskip
\\The key notion for handling this situation on an abstract level -- apart from ad hoc solutions -- was invented by Quillen \cite{Quillen1973} (see B\"uhler \cite{Buehler2009}) and is that of an exact category. Given any additive category, one chooses a so-called exact structure, i.e.~a system of kernel-cokernel pairs satisfying certain axioms. Then, roughly speaking, the pairs in this system behave like the short exact sequences of an abelian category and homological algebra is possible with respect to them; note that the derived category can be defined for each exact category, see Neeman \cite{Neeman}. In any additive category, the set of all split short  exact sequences is an exact structure and by definition it is the smallest one. If the category in addition has kernels and cokernels (i.e.~is preabelian) then there also exists a largest or maximal exact structure (see Sieg, Wegner \cite{SiegWegner2009}) formed by those kernel-cokernel pairs $(f,g)$ in which $f$ is a semi-stable kernel in the notation of 
Richman, Walker \cite{RichmanWalker}, that is in any pushout
\begin{diagram}[height=1.8em,width=2em]
E             & \rTo^{f}         & F            \\
\dTo^{\alpha} & \text{\small PO} & \dTo_{\beta} \\
E'             & \rTo_{f'}         & F'         \\
\end{diagram}
$f'$ is again a kernel, and $g$ is a semi-stable cokernel (which is defined dually). A preabelian category is quasi-abelian and thus admits an intrinsic notion of exactness, cf.~Schneiders \cite{Schneiders1999}, if and only if the maximal exact structure consists of all kernel-cokernel pairs; this is equivalent to the fact that all kernels and cokernels are semi-stable.
\smallskip
\\In the light of the above, for a given preabelian category in which homological methods are to be applied, the first task is to determine the semi-stable (co-)\linebreak{}kernels. In fact, there is a whole zoo of properties between ``(quasi-)abelian'' and ``pre-abelian'' (see~Fig.~1 in Section \ref{Examples}) clarifying this task. So-called left quasi-abelian categories are defined by making the above definition one-sided, i.e.~by requiring all cokernels to be semi-stable. Right quasi-abelian categories are defined dually. A category is semi-abelian if for each morphism the induced morphism between the cokernel of the kernel and the kernel of the cokernel is simultaneously a monomorphism and an epimorphism. Again, left and right semi-abelian categories are defined by making this definition one-sided (see Section \ref{Prep}). Propositions \ref{t1} and \ref{t2} and Theorem \ref{t3} show that the last three properties can also be formulated in terms of inheritance properties under pushouts and pullbacks. 
Therefore, establishing one of these properties for a given category is indeed helpful for understanding its (maximal) exact structure.
\smallskip
\\In Section \ref{Prep}, we specify our notation and present some preparatory results which are rather simple but very useful for our purposes. In Section \ref{MainResult}, we establish a characterization of right and by dualization of left semi-abelian categories via six equivalent conditions, which we combine in our main theorem (Theorem \ref{t3}) to obtain a characterization of semi-abelianity. In the last Section \ref{Examples}, we present two examples illustrating that left semi-abelian categories which are not right semi-abelian (and vice versa) exist and occur in nature.
\smallskip
\\To conclude our introduction, let us mention that the notion of a semi-abelian category in the above sense was invented several times by different mathematicians under different names. Such categories seem to have first appeared in B\u{a}nic\u{a}, Popescu \cite{BanicaPopescu1965} under the name of \textquotedblleft{}cat\'{e}gories pre-ab\'eliennes\textquotedblright{}. At the end of the 1960's, Palamodov \cite{Palamodov1968, Palamodov1971} introduced the same concept under the name of \textquotedblleft{}semi-abelian categories\textquotedblright{} and developed homological algebra within them in order to treat projective and inductive limits of locally convex spaces.  Later, these categories -- again as \textquotedblleft{}semi-abelian categories\textquotedblright{} -- were rediscovered by Rump \cite{Rump2001}, who introduced the notions of left and right semi-abelianity. We refer to \cite[Section 2]{Rump2007} for more historical comments on semi-abelian and quasi-abelian categories and detailed references.
\smallskip
\\We point out that in the current literature the term ``semi-abelian category'' also refers to 
a pointed Barr-exact protomodular category with binary products, see Janelidze, M\`arki, Tholen \cite{JanelidzeMarkiTholen2002}. Such a category is additive if and only if it is abelian.

\section{Notations and Preparatory Results}\label{Prep}\vspace{-5pt}

In the sequel let $\mathcal{A}$ be a preabelian category, i.e.~an additive category with kernels and cokernels. For a morphism $f\colon E\rightarrow F$ in $\mathcal{A}$ we denote by $\ker f\colon \Ker f\rightarrow E$ its kernel and by $\cok f\colon F \rightarrow \Cok f$ its cokernel. Note that kernels and cokernels are unique only up to isomorphisms; we will however speak of \textit{the} kernel and \textit{the} cokernel of a given morphism $f$. According to Richman, Walker \cite{RichmanWalker} we say that $f$ is a \textit{kernel} if there is a morphism $g$ such that $f=\ker g$. \textit{Cokernels} are defined dually. We denote by $\coim f\colon E\rightarrow \Coim f$ the cokernel of $\ker f$ and by $\im f\colon\Im f\rightarrow F$ the kernel of $\cok f$. As above, \textit{image} and \textit{coimage} are unique only up to isomorphims but we will also here use definite articles in the sequel. Then $f$ admits a canonical decomposition $f=(\im f)\circ\bar{f}\circ\coim f$. Following B\u{a}nic\u{a}, Popescu \cite{
BanicaPopescu1965} or Schneiders \cite[Definition 1.1.1]{Schneiders1999}, we say that $f$ is \textit{strict} if $\bar{f}$ is an isomorphism. By definition, $\mathcal{A}$ is abelian if and only if every morphism is strict.
\smallskip
\\According to Rump \cite[p.~167]{Rump2001}, we say that $\mathcal{A}$ is \textit{left semi-abelian} if $\bar{f}$ is a monomorphism for each morphism $f$. Dually, we say that $\mathcal{A}$ is \textit{right semi-abelian} if $\bar{f}$ is an epimorphism for each morphism $f$. If $\mathcal{A}$ is left semi-abelian then each morphism $f$ admits a decomposition $f=i\circ p$ with a cokernel $p$ and a monomorphism $i$. If dually $\mathcal{A}$ is right semi-abelian then each morphism $f$ admits a decomposition $f=i\circ p$ with an epimorphism $p$ and a kernel $i$. In fact, the last assertions are even equivalent to the definitions of left and right semi-abelianity (see Rump \cite[p.~167]{Rump2001}).
\smallskip
\\Left and right semi-abelian categories generalize the concept of what we like to call semi-abelian categories:~$\mathcal{A}$ is semi-abelian if $\bar{f}$ is a bimorphism, i.e.~$\bar{f}$ is a monomorphism and an epimorphism simultaneously for each morphism $f$. By definition, $\mathcal{A}$ is semi-abelian if and only if it is left and right semi-abelian simultaneously. The concept of semi-abelian categories is well-known and was studied by many authors under different names during the last forty-five years; we refer to Section \ref{Introduction} for historical comments and more references.
\smallskip
\\In Section \ref{MainResult}, we will explain that left and right semi-abelian categories admit much more equivalent definitions than the two mentioned above. In order to prove these equivalences we need some preliminary results. The first lemma summarizes well-known facts, see e.g.~Schneiders \cite[Remark 1.1.2]{Schneiders1999}, Richman, Walker \cite[Theorems 1 and 5]{RichmanWalker}, and Kelly \cite[Proposition 5.2]{Kelly1969}.

\begin{lem}\label{l1} Let $\mathcal{A}$ be a preabelian category.\vspace{-12pt}
\begin{compactitem}\vspace{3pt}
\item[(i)] A morphism $f$ is a kernel if and only if $f=\im f$ and it is a cokernel if and only if $f=\coim f$.\vspace{3pt}
\item[(ii)] A morphism $f$ is strict if and only if there is a representation $f=f_1\circ f_0$ with a cokernel $f_0$ and a kernel $f_1$. In every such representation we have
$f_0=\coim f$ and $f_1=\im f$.\vspace{3pt}
\item[(iii)]In every pullback
\begin{diagram}[height=1.8em,width=2em]
P          & \rTo^{p_G}       & G        \\
\dTo^{p_E} & \text{\small PB} & \dTo_{t} \\
E          & \rTo_{f}         & F        \\
\end{diagram}
we have $\ker f= p_E\circ \ker p_G$. If $f$ is the kernel of a morphism $h$ then $p_G$ is the kernel of $h\circ{}t$. If $f$ is a monomorphism then so is $p_G$.
\smallskip
\\Dually, in every pushout
\begin{diagram}[height=1.8em,width=2em]
& E              & \rTo^{g}         & F          \\
\hspace{8pt}& \dTo^{s}       & \text{\small PO} & \dTo_{s_F} \\
& G              & \rTo_{s_G}       & S          \\
\end{diagram}
we have $\cok g=(\cok s_G)\circ s_F$. If $g$ is the cokernel of a morphism $h$ then $s_G$ is the cokernel of $s\circ h$. If $g$ is an epimorphism then so is $s_G$.\hfill\qed
\end{compactitem}
\end{lem}

\vspace{-10pt}
Let us stress that the last statement of Lemma \ref{l1} implies that in any preabelian category kernels pullback to kernels and cokernels pushout to cokernels.
\smallskip
\\The proof of the next lemma was inspired by an idea of Yakovlev \cite[Lemma 1]{Yak}.

\begin{lem}\label{l2} Let $\mathcal{A}$ be a preabelian category and let $f\colon E\rightarrow F$ and $g\colon F\rightarrow G$ be morphisms. Then we have $\coker(g\circ\im f)=\coker(g\circ f)$ and $\ker((\coim g)\circ f)=\ker(g\circ f)$.
\end{lem}\vspace{-20pt}
\begin{proof} To prove that $\coker(g\circ \im f)$ is the cokernel of $g\circ{}f$, we first observe $\coker(g\circ\im f)\circ g\circ f=0$. Given a morphism $x$ with $x\circ{}g\circ{}f=0$, we get a morphism $\rho$ such that $x\circ g=\rho\circ{}\coker f$. 
This yields $x\circ{}g\circ{}\im f=0$ and hence $x=y\circ \coker(g\circ{}\im f)$ for some unique $y$. The second assertion is obtained by duality.
\end{proof}

\vspace{-10pt}

We need the following corollary for our proofs in Section \ref{MainResult}.

\begin{cor}\label{corl1} Let $\mathcal{A}$ be a preabelian category.\vspace{-10pt}
\begin{compactitem}
\item[(i)] Assume that the kernels of $\mathcal{A}$ are stable under composition. Let $f$ be an arbitrary morphism and $g$ be a kernel such that the composition $g\circ{}f$ is defined. Then $\im(g\circ f)=g\circ\im f$.\vspace{3pt}
\item[(ii)] Assume dually that the cokernels of $\mathcal{A}$ are stable under composition. Let $g$ be an arbitrary morphism and $f$ be a cokernel such that the composition $g\circ f$ is defined. Then $\coim(g\circ{}f)=(\coim g)\circ{}f$.
\end{compactitem} 
\end{cor}\vspace{-20pt}
\begin{proof} (i) By Lemma \ref{l2} we have $\im(g\circ f)=\ker\cok(g\circ f)=\ker\cok(g\circ\ker(\cok f))$. Since $g\circ\ker\cok f$ is a kernel, Lemma \ref{l1}.(i) yields the desired equality.
Item (ii) follows by duality.
\end{proof}

\section{Main Results}\label{MainResult}\vspace{-5pt}

As already announced, Proposition \ref{t1} below contains six equivalent definitions of right semi-abelianity for a preabelian category. Let us mention at this point that the implications \textquotedblleft{}(i)$\Rightarrow$(ii)$\Rightarrow$(iii)$\Rightarrow$(iv)$\Rightarrow$(v)$\Rightarrow$(vi)\textquotedblright{} in Proposition \ref{t1} and even the equivalence of (i), (iv) and (v) were in fact established by Kuz$'$minov, {\v{C}}herevikin \cite[Lemmas 4, 5, 6 and Theorem 1]{KCh}. Moreover, Grandis \cite{Grand91-1} proved the equivalence of (i) and (vi) in the more general (non-additive) context of so-called \textquotedblleft{}ex2 categories\textquotedblright{}. In~\cite[Proposition~1]{Rump2001}, Rump proved the equivalence of (i), (iii) and (iv). Finally, the proof of the implication \textquotedblleft{}(iv)$\Rightarrow$(vii)\textquotedblright{} was given by Kopylov, Kuz$'$minov \cite{KoK00,KoK09}. In the sequel, we give a complete proof based on our preparations in Section \ref{Prep} and thus in particular 
obtain a unified version of the work cited above.
\smallskip
\\Given a commutative square
\begin{diagram}[height=1.8em,width=2em]
C             & \rTo^{g}         & D            \\
\dTo^{\alpha} & \text{\small(1)} & \dTo_{\beta} \\
A             & \rTo_{f}         & B            \\
\end{diagram}
in a preabelian category $\mathcal{A}$, we denote by $\hat\alpha\colon\Ker g\rightarrow\Ker f$ the unique morphism satisfying $\alpha\circ\ker g=(\ker f)\circ\hat\alpha$ and dually by $\hat\beta\colon\Coker g\rightarrow\Coker f$ the unique morphism satisfying $(\coker f)\circ\beta=\hat\beta\circ\coker g$. The above square will be used throughout the remainder of the section without repetition. 

\begin{prop}\label{t1} Let $\mathcal{A}$ be preabelian. Then the following are equivalent.\vspace{-12pt}
\begin{compactitem}\vspace{1pt}
\item[(i)] \hspace{-1.75pt}$\mathcal{A}$ is right semi-abelian.\vspace{1pt}
\item[(ii)] \hspace{-1.75pt}If $h\circ{}l$ is a kernel then so is $l$.\vspace{1pt}
\item[(iii)] \hspace{-1.75pt}If (1) is a pushout and $g$ is a kernel then (1) is a pullback.\vspace{1pt}
\item[(iv)] \hspace{-1.75pt}If (1) is a pushout and $g$ is a kernel then $f$ is a monomorphism.\vspace{1pt}
\item[(v)] \hspace{-1.75pt}If (1) is a pushout, $g$ is a kernel and $\beta$ is a cokernel then $f$ is a monomorphism.\vspace{1pt}
\item[(vi)] \hspace{-1.75pt}If $l$ and $h$ are kernels and $h\circ{}l$ is defined then $h\circ{}l$ is a kernel.\vspace{1pt}
\item[(vii)] \hspace{-1.75pt}If (1) is a pushout and $g$ is strict then $\hat\alpha$ is an epimorphism.
\end{compactitem}\vspace{-12pt}
\end{prop}
\begin{proof}\textquotedblleft{}(i)$\Rightarrow$(ii)\textquotedblright{} Suppose that $h\circ{}l$ is a kernel. Let $l=i\circ p$ with $i$ a kernel and $p$ an epimorphism. Put $\alpha=\coker(h\circ{}l)$, then $h\circ{}l=\ker\alpha$ by Lemma \ref{l1}.(i). It is easy to check that $h\circ{}i$ is a monomorphism and thus the kernel of $\alpha$. Therefore, $p$ is an isomorphism and consequently $l$ is a kernel.
\smallskip
\\\textquotedblleft{}(ii)$\Rightarrow$(iii)\textquotedblright{} Suppose that (1) is a pushout and $g$ is a kernel. Then $\cb{f,}{-\beta}=\cok\bc{\alpha}{g}$ and $\bc{\alpha}{g}$ is a kernel by hypothesis. Lemma \ref{l1}.(i) provides $\bc{\alpha}{g}=\ker\,\cb{f,}{-\beta}$. Hence, the square (1) is a pullback.
\smallskip
\\\textquotedblleft{}(iii)$\Rightarrow$(iv)\textquotedblright{} Suppose that (1) is a pushout and $g$ is a kernel. Then (1) is a pullback by hypothesis. If $x$ is a morphism with $f\circ x=0$ then the fact that the square is a pullback implies the existence of a morphism $y$ with $\alpha\circ{}y=x$ and $g\circ{}y=0$. Since $g$ is a monomorphism, $y=0$ holds. Consequently, $x=0$ and $f$ is a monomor-\linebreak{}phism.
\smallskip
\\\textquotedblleft{}(iv)$\Rightarrow$(v)\textquotedblright{} Trivial. 
\smallskip
\\\textquotedblleft{}(v)$\Rightarrow$(vi)\textquotedblright{} Put $\alpha =\coker l$, $\delta =\coker h$ and consider a pushout $\gamma\circ{}\alpha =\beta\circ{}h$. By Lemma \ref{l1}.(iii), $\beta$ is a cokernel. It is easy to check that $\beta=\coker(h\circ{}l)$. By hypothesis, $\gamma$ is a monomorphism and we have to show that $h\circ{}l=\ker\beta$. Let $\lambda$ be a morphism such 
that $\beta\circ{}\lambda=0$. Consider the diagram
\begin{diagram}[height=1.9em,width=2em]
      &                                       & \cdot         &                                        &              &                                                      &       \\
      & \ldDash(2,2)^{\rotatebox{-35}{$\nu$}} & \dDash^{\mu}  & \rdTo(2,2)^{\rotatebox{35}{$\lambda$}} &              &                                                      &       \\
\cdot & \rTo_l                                & \cdot         & \rTo_h                                 & \cdot        & \rTo^{\delta}                                        & \cdot \\
      &                                       & \dTo^{\alpha} &                                        & \dTo_{\beta} & \ruTo(2,2)_{\;\;\rotatebox{-43}{$\cok\gamma$}\!\!\!} &       \\
      &                                       & \cdot         & \rTo_{\gamma}                          & \cdot        &                                                      &       \\
\end{diagram}
where the arrows $\nu$ and $\mu$ are introduced below and the relation $\delta=(\cok\gamma)\circ\beta$ results from Lemma \ref{l1}.(iii). By the above, $\beta\circ{}\lambda=0$ implies $\delta\circ{}\lambda=0$. Lemma \ref{l1}.(i) implies $h=\ker\delta$ and therefore there is a morphism $\mu$ with $h\circ{}\mu=\lambda$. Since $\gamma\circ{}\alpha\circ{}\mu =\beta\circ{}\lambda=0$ and $\gamma$ is a monomorphism, we obtain $\alpha\circ{}\mu=0$. Lemma \ref{l1}.(i) yields $l=\ker\alpha$ and hence there exists a morphism $\nu$ such that $l\circ\nu =\mu$. Now, $h\circ{}l\circ{}\nu =h\circ{}\mu=\lambda$ and since $l$ and $h$ are monomorphisms, $\nu$ is unique with this property. This proves that $h\circ{}l=\ker\beta$, i.e.~$h\circ{}l$ is a kernel.
\smallskip
\\\textquotedblleft{}(vi)$\Rightarrow$(i)\textquotedblright{} Let kernels be stable under composition and let $h$ be a morphism with the canonical decomposition $h=(\im h)\circ\bar{h}\circ\coim h$. Corollary \ref{corl1}.(i) yields $\im h=(\im h)\circ\im(\bar{h}\circ\coim h)=(\im h)\circ\im\bar{h}$. This implies that $\im\bar{h}=\id$, i.e.~$\bar{h}$ is an epimorphism.
\smallskip
\\\textquotedblleft{}(iv)$\Rightarrow$(vii)\textquotedblright{} Assume that $g=g_1\circ{}g_0$ is a representation of the strict morphism $g$ with $g_1$ a kernel
and $g_0$ a cokernel. Consider the pushout
\begin{diagram}[height=1.8em,width=2em]
C             & \rTo^{g_0}            & A'            \\
\dTo^{\alpha} & \scriptstyle\text{PO} & \dTo_{\beta'} \\
A             & \rTo_{f_0}            & B'            \\
\end{diagram}
and compute $\beta\circ{}g_1\circ{}g_0=f\circ{}\alpha$. Thus, there exists a unique morphism $f_1\colon B'\rightarrow B$ with $f=f_1\circ{}f_0$ and $\beta\circ{}g_1=f_1\circ{}\beta'$. It is easy to see (cf.~\cite[Lemma 5.1]{Kelly1969}) that
\begin{diagram}[height=1.8em,width=2em]
A'             & \rTo^{g_1}            & D            & \\
\dTo^{\beta'}  & \scriptstyle\text{PO} & \dTo_{\beta} &\hspace{-7pt} \\
B'             & \rTo_{f_1}            & B            & \\
\end{diagram}
is again a pushout and hence $f_1$ is a monomorphism. By hypothesis, $g_1$ is a monomorphism. Let $x\colon\Ker f\rightarrow X$ be a morphism such that $x\circ\hat\alpha=0$. In the pushout
\begin{diagram}[height=1.9em,width=1.8em]
\Ker f        & \rTo^{x}                   & X          & \\
\dTo^{\ker f} & \scriptstyle\text{\!PO\;}  & \dTo_{z_2} &\hspace{22pt} \\
A             & \rTo_{z_1}                 & Z          & \\
\end{diagram}
we have $z_1\circ\alpha\circ\ker g=z_1\circ(\ker f)\circ\hat\alpha = z_2\circ{}x\circ{}\hat\alpha=0$ and $g_0=\coker\ker g$ by Lemma \ref{l1}.(ii). Thus, there exists a unique morphism $w\colon A'\rightarrow Z$ such that $z_1\circ\alpha=w\circ{}g_0$. Since $\beta'\circ{}g_0=f_0\circ{}\alpha$ is a pushout, there is a unique morphism $\sigma\colon B'\rightarrow Z$ such that $w=\sigma\circ{}\beta'$ and $z_1=\sigma\circ{}f_0$. But now $z_2\circ{}x=z_1\circ{}\ker f=\sigma\circ{}f_0\circ{}\ker f=0$ because $0=f\circ\ker f=f_1\circ{}f_0\circ\ker f$ and $f_1$ is a monomorphism. Since $\ker f$ is a kernel, $z_2$ is a monomorphism, whence $x=0$ and $\hat\alpha$ is an epimorphism.
\smallskip
\\\textquotedblleft{}(vii)$\Rightarrow$(iv)\textquotedblright{} If (1) is a pushout with $g$ a kernel then $(\ker f)\circ\hat\alpha = \alpha\circ\ker g=0$. Since $\hat\alpha$ is an epimorphism, this yields $\ker f=0$, i.e.~$f$ is a monomorphism.
\end{proof}

\vspace{-10pt}
We can now formulate the statement dual to Proposition \ref{t1}.

\begin{prop}\label{t2} Let $\mathcal{A}$ be preabelian. Then the following are equivalent.\vspace{-11pt}
\begin{compactitem}\vspace{1pt}
\item[(i$'$)\hspace{-1pt}] $\mathcal{A}$ is left semi-abelian.\vspace{1pt}
\item[(ii$'$)\hspace{-1pt}] If $h\circ l$ is a cokernel then so is $h$.\vspace{1pt}
\item[(iii$'$)\hspace{-1pt}] If (1) is a pullback and $f$ is a cokernel then (1) is a pushout.\vspace{1pt}
\item[(iv$'$)\hspace{-1pt}] If (1) is a pullback and $f$ is a cokernel then $g$ is an epimorphism.\vspace{1pt}
\item[(v$'$)\hspace{-1pt}] If (1) is a pullback, $f$ is a cokernel and $\alpha$ is a kernel then $g$ is an epimorphism.\vspace{1pt}
\item[(vi$'$)\hspace{-1pt}] If $l$ and $h$ are cokernels and $h\circ{}l$ is defined then $h\circ{}l$ is a cokernel.
\item[(vii$'$)] \hspace{-1.75pt}If (1) is a pullback and $f$ is strict then $\hat\beta$ is a monomorphism.\hfill\qed
\end{compactitem} 
\end{prop}

\vspace{-10pt}
Combining Propositions \ref{t1} and \ref{t2}, we obtain the following characterization of semi-abelianity.

\begin{thm}\label{t3} Let $\mathcal{A}$ be a preabelian category $\mathcal{A}$. Then $\mathcal{A}$ is semi-abelian if and only if it satisfies one of conditions \ref{t1}.(i)--(vii) and one of conditions \ref{t2}.(i$'$)--(vii$'$).\hfill\qed
\end{thm}

\vspace{-15pt} 

\section{Examples}\label{Examples}\vspace{-10pt}

The results of Section \ref{MainResult} trigger the question if there are natural examples of categories which are right semi-abelian but not left semi-abelian and vice versa. In this section we present two examples of this type and then reclassify these examples in the context of left and right quasi-abelian categories in the sense of Rump \cite[p.~168]{Rump2001}.
\smallskip
\\In what follows, by a \textquotedblleft{}space\textquotedblright{} we mean a locally convex space whose topology is not necessarily Hausdorff. The category formed by all these spaces with linear and continuous maps as morphisms is denoted by LCS. For notation concerning the theory of locally convex spaces we refer to \cite{FloretWloka, Jarchow, BPC}. The full subcategory of Hausdorff spaces is denoted by HD-LCS. Our first example is the category HD-BOR of Hausdorff locally convex spaces that are bornological (cf.~\cite[Example 4.2]{SiegWegner2009}), i.e.~a linear map is continuous whenever it is bounded on bounded sets. Every metrizable space is bornological and this class is stable under the formation of quotients and locally convex inductive limits. Let BOR be the full subcategory of LCS consisting of bornological spaces (cf.~\cite[\textsection{}\,23, 1.5 and \textsection{}\,11, 2.]{FloretWloka}). BOR is additive and by \cite[\textsection{}\,23, 2.9]{FloretWloka} for a morphism $f\colon E\rightarrow F$ 
in BOR the cokernel is the space $F/f(E)$, i.e.~the cokernels are the same as in LCS. The kernel of $f$ is the space $f^{-1}(0)^{\text{BOR}}$, i.e.~the linear space $f^{-1}(0)$ endowed with the bornological topology associated with the topology induced by $E$, see \cite[\textsection{}\,11, 2.2]{FloretWloka}. In \cite[Example 4.1]{SiegWegner2009}, it was pointed out that BOR is an example of a category which is semi-abelian but not quasi-abelian; the first property is easy to check in view of the above, whereas the absence of quasi-abelianity follows from an example due to Bonet, Dierolf \cite{BonetDierolf2005}.
\smallskip
\\Let us now start with an example of a category which is left semi-abelian but not right semi-abelian. The following relies on \cite[Example 4.2]{SiegWegner2009}.

\begin{ex}\label{HD-BOR} By HD-BOR we denote the full subcategory of HD-LCS consisting of bornological spaces. Clearly, HD-BOR is a full subcategory of BOR. From \cite[Example 4.1]{SiegWegner2009} it follows that the kernels in HD-BOR are those of BOR. This is not true for cokernels because a quotient space is Hausdorff if and only if the corresponding subspace is closed. Let $f\colon E\rightarrow F$ be a morphism in HD-BOR. Then the cokernel of $f$ is the space $F/\overline{f(E)}$ endowed with the quotient topology. Hence, the coimage of $f$ is the space $E/f^{-1}(0)$ endowed with the quotient topology and the image of $f$ is $\overline{f(E)}^{\scriptscriptstyle\text{BOR}}$, the closure of $f(E)$ in $F$ but endowed with the finer associated bornological topology (it is thus not clear whether $f(E)$ is still dense in this space).
\smallskip
\\Using an example due to Grothendieck \cite{Grothendieck1954} (cf.~Bonet, P\'{e}rez Carreras \cite[8.6.12]{BPC}) it was shown in \cite[Example 4.2]{SiegWegner2009} that there is a morphism $f\colon E\rightarrow F$ in HD-BOR such that $\bar{f}\colon \coim f\rightarrow \im f$ is not an epimorphism, whence HD-BOR is not right semi-abelian. On the other hand, for each morphism $f$ the induced morphism $\bar{f}$ is easily seen to be injective whence $\bar{f}$ is a monomorphism and therefore HD-BOR is left semi-abelian.
\end{ex}
\vspace{-10pt}
Let us note that the forementioned example of Grothendieck arises in the framework of so-called well-located and limit subspaces. These properties of subspaces of locally convex inductive limits arise naturally if one deals with the question whether partial differential or convolution operators on some space are surjective. We refer to Floret \cite{Floret} for precise definitions, explanations and further references.
\smallskip
\\The next example gives a category which is right semi-abelian but not left semi-abelian; our methods strongly rely on Sieg \cite{SiegThesis}.

\begin{ex}\label{HD-COM} Let HD-COM be the full subcategory of HD-LCS consisting of the complete locally convex spaces. Since finite products of complete spaces are again complete, HD-COM is additive. For a morphism $f\colon E\rightarrow F$ in HD-COM, the kernel is the space $f^{-1}(0)$ endowed with the topology induced by $E$; since we are dealing with Hausdorff spaces, $f^{-1}(0)$ is closed in $E$ and hence complete. The cokernel of $f$ is the Hausdorff completion $C(F/f(E))$, i.e.~the completion of $F/\bar{f(E)}$. Therefore, the cokernel of the kernel of $f$ is the space $C(E/f^{-1}(0))$ and the kernel of the cokernel is the space $\bar{f(E)}$ endowed with the topology induced by $F$. It is easy to see that the map $\bar{f}\colon C(E/f^{-1}(0))\rightarrow\bar{f(E)}$ has dense range and is thus an epimorphism for each morphism $f$ in HD-COM. Therefore, HD-COM is right semi-abelian. On the other hand, one can copy the proof of Sieg \cite[Proposition 3.1.6]{SiegThesis} verbatim to construct a morphism $f$ 
such that $\bar{f}$ is not a monomorphism. This shows that HD-COM is not left semi-abelian.
\end{ex}
\vspace{-14pt}
As announced at the beginning of this section, we finish by reclassifying the two examples in the context of (left and right) quasi-abelian categories. For these notions, the reader is referred to \cite{Schneiders1999} and \cite{Rump2001}; note that Rump uses the term \textquotedblleft{}almost abelian\textquotedblright{} instead of \textquotedblleft{}quasi-abelian\textquotedblright{}. Using \cite[Corollary 1.2 and Proposition 1.3]{Rump2001}, we obtain a whole zoo of implications among the notions discussed so far, see Fig.~1.

\begin{center}
\begin{picture}(322,120)(0,0)
\put(0,0){\psfig{figure=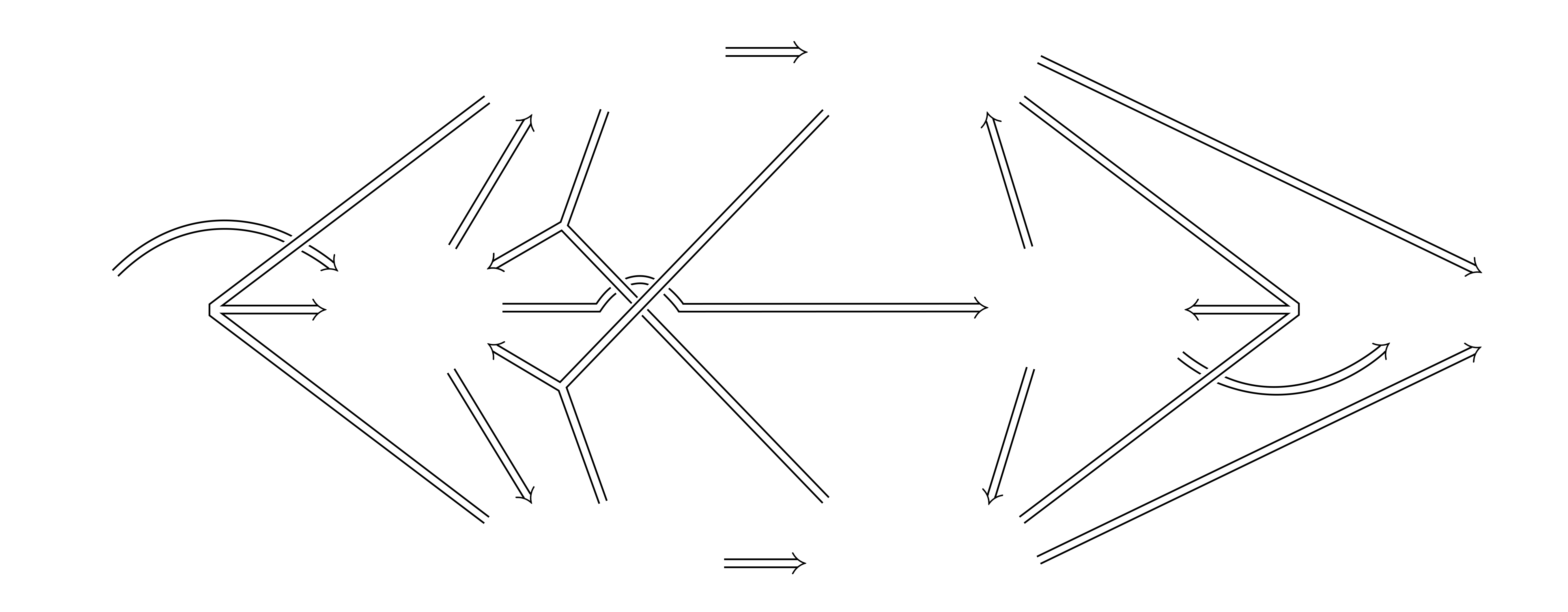, width=11.5cm}}
         \put(107.5,119){\footnotesize left quasi-}                                 \put(176,119){\footnotesize left semi-}
         \put(113,110){\footnotesize abelian}                                       \put(178.8,110){\footnotesize abelian}

                       \put(75,64.7){\footnotesize quasi-}                \put(217,64.7){\footnotesize semi-}

\put(4.4,60.5){\footnotesize abelian}                                                                \put(283,60.5){\footnotesize preabelian}

                       \put(72,56){\footnotesize abelian}                  \put(213,56){\footnotesize abelian}

         \put(106,13){\footnotesize right quasi-}                                   \put(172.5,13){\footnotesize right semi-}
         \put(113,3.5){\footnotesize abelian}                                       \put(178.8,3.5){\footnotesize abelian}
\put(32,-15){\small Figure 1: Zoo of properties between abelian and preabelian.}
\end{picture}
\end{center}
\smallskip
It is well known that there are categories which are quasi-abelian but not abelian: For instance the category LCS, the category HD-LCS, the category of Banach spaces or the category of Fr\'{e}chet spaces are quasi-abelian but not abelian.  We refer to Prosmans \cite{Prosmans2000} for more details.
\smallskip
\\The question if the notions of quasi-abelian and semi-abelian categories coincide in general is known as Raikov's Conjecture and was solved in the negative by the counterexample in \cite{BonetDierolf2005} which shows that BOR is semi-abelian but not quasi-abelian (cf.~our remarks before Example \ref{HD-BOR} and \cite[Example 4.1]{SiegWegner2009}). Rump \cite{Rump2007} gave another counterexample in the context of representation theory. In fact, the category BOR also demonstrates that the implications \textquotedblleft{}left quasi-abelian $\Rightarrow$ left semi-abelian\textquotedblright{} and  \textquotedblleft{}right quasi-abelian $\Rightarrow$ right semi-abelian\textquotedblright{} can in general not be reverted: the counterexample of Bonet, Dierolf \cite{BonetDierolf2005} in fact shows that BOR is not left quasi-abelian. On the other hand, it follows directly that BOR also fails to be right quasi-abelian (assume that this is true, then \cite[Proposition 1.3]{Rump2001} would imply that BOR is even quasi-
abelian). Examples \ref{HD-BOR} and \ref{HD-COM} show that all three implications starting at \textquotedblleft{}semi-abelian\textquotedblright{} and also the three implications 
ending at \textquotedblleft{}preabelian\textquotedblright{} cannot be reverted. 
\medskip
\\Let us point out that the category of PLS-spaces studied by Sieg \cite{SiegThesis} has properties similar to those of HD-COM and that the former category has important applications in analysis (see \cite{SiegThesis} for details).
\smallskip
\\To conclude, let us mention that it recently turned out that the category of LB-spaces, i.e.~countable inductive limits of Banach spaces, is an example for a category which is left quasi-abelian but not right quasi-abelian (and thus necessarily not semi-abelian). The latter was discovered by Dierolf \cite{BD} together with the second-named author.

\begin{ack}{\small Both authors would like to thank the referee for several useful comments which helped to improve this article. The first-named author was partially supported by the Russian Foundation for Basic Research (Grant 09-01-00142-a), the State Maintenance Program for the Leading Scientific Schools and Junior Scientists of the Russian Federation (NSh-6613.2010.1), and the Integration Project \textquotedblleft{}Quasiconformal Analysis and Geometric Aspects of Operator Theory\textquotedblright{} of the Siberian Branch and the Far East Branch of the Russian Academy of Sciences.}
\end{ack}

{\bf \textsf{Authors' Addresses:}}
\smallskip
\\{\small Ya.~Kopylov, Sobolev Institute of Mathematics, Pr.~Akademika Koptyuga 4, 630090 Novosibirsk, RUSSIA, e-mail: yakop@math.nsc.ru.
\smallskip
\\S.-A.~Wegner, Fachbereich C, Bergische Universit\"at Wuppertal, Gau\ss{}stra\ss{}e 20, D-42097 Wuppertal, GERMANY, e-mail: wegner@math.uni-wuppertal.de.}

\end{document}